\documentclass[english,reqno]{amsart}

\usepackage{amsmath, appendix, ulem}
\usepackage[nobysame]{amsrefs}
\usepackage{amssymb, color}
\usepackage[margin=1in]{geometry}
\usepackage{mathrsfs}
\usepackage{graphicx}
\usepackage{subfig}
\usepackage{float}
\usepackage{epsf}
\usepackage{hyperref}
\usepackage{titletoc}

\usepackage{subeqnarray}
\usepackage{cases}
\graphicspath{{../Figures/}}

\numberwithin{equation}{section}

\newtheorem{lem}{Lemma}[section]
\newtheorem{thm}{Theorem}[section]

\theoremstyle{remark}
\newtheorem{rmk}{Remark}[section]

\newtheorem{assum}{Assumption}

\newcommand{\nn}{\nonumber}
\newcommand{\R}{{\mathbb R}}

\renewcommand{\hat}{\widehat}

\newcommand{\mc}[1]{\mathcal{#1}}

\newcommand{\EE}{\mathbb{E}}
\newcommand{\RR}{\mathbb{R}}
\newcommand{\NN}{\mathbb{N}}
\newcommand{\PP}{\mathbb{P}}

\newcommand{\VG}{V^\gamma}

\newcommand{\XG}{X^\gamma}

\newcommand{\fm}{f^\gamma}
\newcommand{\rom}{\rho^\gamma}

\newcommand{\rd}{\mathrm{d}}

\newcommand{\norm}[1]{\left\lVert#1 \, \right\rVert}

\author{Hui Huang}
\address{Department of Mathematics and Statistics, University of Calgary, Calgary, Canada}
\email{hui.huang1@ucalgary.ca}

\date{\today}
\thanks{H. H. is partially supported by the Pacific Institute for the Mathematical Sciences (PIMS) postdoc fellowship.}
\begin{document}
\title[Overdamped limit]{Quantitative estimate of the overdamped limit for the Vlasov--Fokker--Planck systems}
\maketitle
\begin{abstract}This note adapts a probabilistic approach to establish a quantified estimate of the overdamped limit for the Vlasov--Fokker--Planck equation towards the aggregation-diffusion equation, which in particular includes cases of the Newtonian type singular forces.  The proofs are based on the investigation of  the weak convergence of the corresponding stochastic differential equations (SDEs) of Mckean type in the continuous path space. We show that one can obtain the same convergence rate as in \cite{choi2020quantified} under the same assumptions.
\end{abstract}

\section{Introduction}
The present note is concerning with the following kinetic Vlasov-Fokker-Planck (VFP) equation
\begin{equation}\label{VFPeq}
\partial_{t} \fm_t+\gamma v \cdot \nabla_{x} \fm_t+\gamma\nabla_v\cdot(F(x,\rom_t)\fm_t)=\gamma^2\nabla_v\cdot(\nabla_v\fm_t+v\fm_t),\quad \fm|_{t=0}=f_0 \,,
\end{equation}
in dimension $d\geq 1$, where $(\fm_t)_{t\geq 0}\subset \mc{P}(\RR^{d}\times \RR^d)$ is a family of probability measures on $\RR^{d}\times \RR^d$ and $\gamma>0$ is the damping coefficient. Here $\rom_t=\int_{\RR^d}\fm_t(\cdot,dv)\in \mc{P}(\RR^d)$ represents the spacial distribution, namely the $x$-marginal of $\fm_t$.  In the sequel we may abuse the notations for a measure  and its Lebesgue density for simplicity.
Moreover, we assume the driving force $F$ is arising from an external potential $\Phi:\RR^d\to\RR$ and/or interaction potential $K:\RR^d\to\RR$, which is of the following from
\begin{equation}
F(x,\rho)=-\nabla\Phi(x)-(\nabla  K\ast\rho)(x)\quad \mbox{for }(x,\rho)\in\RR^d\times \mc{P}(\RR^d)\,,
\end{equation}
where 
\begin{equation}
(\nabla  K\ast\rho)(x):=\int_{\RR^d}\nabla K(x-y)\rho(dy)\,.
\end{equation}

It is well-known that the VFP equation \eqref{VFPeq} can be derived from a system of large number of particles interacting through the force field $F$, which satisfies the following system of stochastic differential equations
\begin{align}\label{particle system}
\begin{cases}
dX_t^{i,\gamma}= \gamma V_t^{i,\gamma}d t, \\
dV_t^{i,\gamma}=-\gamma\nabla\Phi(X_t^{i,\gamma})dt-\frac{\gamma}{N}\sum_{j\neq i}\nabla K(X_t^{i,\gamma}-X_t^{j,\gamma})dt -\gamma^2 V_t^{i,\gamma}+\sqrt{2}\gamma dB_t^i,\quad i=1,\cdots,N\,,
\end{cases}
\end{align}
where $X_t^{i,\gamma},~V_t^{i,\gamma}\in \RR^d$ denote the position and velocity of the $i$-th particle at time $t$,  and $\{(B_t^i)_{t\geq0}\}_{i=1}^N$ are $N$ independent $d$-dimensional Brownian motions. Here we assume the initial data $\{(X_0^{i,\gamma},V_0^{i,\gamma})\}_{i=1}^N$ are i.i.d. with the common distribution $f_0$.  Model \eqref{VFPeq} and its microscopic counterpart \eqref{particle system} have been widely used in the investigation of complex systems that
model collective behaviour (or swarming), an area that has attracted a great deal of attention, see for instance \cite{carrillo2010asymptotic,bellomo2011modeling,ha2008particle,cucker2007emergent,motsch2014heterophilious} and references therein. Note that equation \eqref{VFPeq} also includes the classical Vlasov--Poisson--Fokker--Planck system when $\nabla K=a \frac{x}{|x|^d}$, $d\geq 3$. The case $a > 0$ corresponds, for example, to the electrostatic (repulsive) interaction of charged particles in a plasma, while the case $a < 0$ describes the attraction between massive particles subject to gravitation in astrophysics.

Under suitable assumption on $\Phi$ and $K$, as $N\to\infty$, the mean-field limit result, see for example \cite{bolley2011stochastic,huang2020mean,carrillo2019propagation,sznitman1991topics,jabin2017mean,lazarovici2017mean,fetecau2019propagation,liu2019propagation}, shall show that the particle dynamics \eqref{particle system} well approximates the following  mean-field  nonlinear Mckean process
\begin{subequations}\label{2MVeq}
	\begin{numcases}{}
	dX_t^\gamma= \gamma V_t^\gamma dt, \label{eqX}\\
	dV_t^\gamma =-\gamma^2  V_t^\gamma dt+\gamma F(X_t^\gamma,\rho_t^\gamma)dt+\sqrt{2}\gamma dB_t\,, \label{eqV}
	\end{numcases}
\end{subequations}
where  the initial data $(X_0,V_0)$ is the same as in \eqref{particle system}. Here  $\rom_t=\mbox{Law} (\XG_t)$, the $x$-marginal of  $\fm_t=\mbox{Law}(X_t^\gamma ,V_t^\gamma )$, which makes the set of equations \eqref{2MVeq} nonlinear. A direct application of It\^{o}'s formula, the law $f_t^\gamma :=f^\gamma(t,\cdot,\cdot)$ at time $t$ is a weak solution to the following  
with the initial data $\fm_0(x,v)=\mbox{Law}(X_0,V_0)$.

In this note, we are interested in the VFP equation \eqref{VFPeq} in the overdamped regime, namely in the regime where $\gamma\gg 1$. When $\gamma \to \infty$, it is expected that the kinetic equation \eqref{VFPeq}  will converge to the following so-called aggregation-diffusion equation
\begin{equation}\label{adeq}
\partial_t\rho_t+\nabla_x\cdot(\rho_t F(x,\rho_t))=\Delta_x\rho_t,\quad \rho|_{t=0}=\rho_0\,.
\end{equation}
Equation of the type \eqref{adeq} appears in various contexts, such as biological aggregations \cite{topaz2006nonlocal}, material science and granular media \cite{toscani2000one}, self-assembly of nanoparticles \cite{holm2005aggregation} and molecular dynamics simulations of matter \cite{haile1993molecular}. The most noble example is the case when $\Delta K=\delta_0$, which is corresponding to the Keller--Segel model  for chemotaxis \cite{keller1970initiation}.  Similar to \eqref{2MVeq}, we have the underlying nonlinear Mckean process satisfying
\begin{align}\label{1MVeq}
X_t=X_0+\int_0^tF(X_s,\rho_s)ds +\sqrt{2}B_t
\end{align}
with $\rho_t=\mbox{Law}(X_t)$, and it satisfies the aggregation-diffusion equation \eqref{adeq}.  Since the well-posedness of equations \eqref{VFPeq} and \eqref{adeq}, and nonlinear processes \eqref{2MVeq}  and \eqref{1MVeq}  are not the focus of the present note, we refer readers to, for instance \cite{bouchut1993existence,sznitman1991topics,bolley2011stochastic,carrillo2019aggregation,godinho2015propagation},  for more discussions on the topic of solvability.

In the absence of the interaction potential, i.e. $K=0$, the overdamped limit was first formally discussed in \cite{kramers1940brownian} by Kramers through introducing a coarse-graining map. Since then, more related results have been proven by using stochastic and asymptotic techniques \cite{freidlin2004some,hottovy2012noise}, or variational methods \cite{duong2018quantification}. In the presence of the interaction potential $K$, a variational technique was proposed in \cite{duong2017variational} without obtaining the convergence rate. Most recently the authors in \cite{choi2020quantified} obtained a quantified overdamped limit in 2-Wasserstein distance for the VFP equation with nonlocal forces. 
In the absence of diffusion, similar problems were also investigated in \cite{carrillo2020quantitative,jabin2000macroscopic,fetecau2015first} via large friction limit. Especially in \cite{carrillo2020quantitative}, the authors obtained a quantitative convergence rate, and it was extended to the case with diffusion in \cite{carrillo2021large}.  In the present note we will use a different alternative approach to obtain the quantified overdamped limit under the same assumptions as in \cite{choi2020quantified}.  Instead of looking into the PDEs \eqref{VFPeq} and \eqref{adeq} directly, we will investigate their underlying Mckean processes \eqref{2MVeq} and \eqref{1MVeq} in the continuous path space. This is less technical than the methods of PDE analysis and is more intuitive in way, and hopefully more
 accessible to non-specialists. Such method has been used in \cite{hui2021} to obtain the consensus based optimization from the particle swarm optimization with the limit of zero inertia.

\textbf{Outline of the proof.}
Let us  first solve $V_t^\gamma $ from \eqref{eqV} and obtain that
\begin{equation}
V_t^\gamma =e^{-\gamma^2 t}V_0+\gamma\int_0^te^{-\gamma^2(t-s)}F(X_s^\gamma,\rho_s^\gamma )ds+\sqrt{2}\gamma \int_0^te^{-\gamma^2(t-s)}dB_s\,,
\end{equation}
which implies that
\begin{align}\label{onlyX}
X_t^\gamma &=X_0+\gamma\int_0^tV_\tau^\gamma d\tau=X_0+\gamma\int_0^t  e^{-\gamma^2\tau}V_0d\tau
+\gamma^2\int_0^t\int_0^\tau e^{-\gamma^2(\tau-s)}F(X_s^\gamma,\rho_s^\gamma )dsd\tau\notag\\
& \quad +\sqrt{2}\gamma^2\int_0^t\int_0^\tau e^{-\gamma^2(\tau-s)}dB_s d\tau\nn\\
&=X_0+\frac{1}{\gamma}(1-e^{-\gamma^2t})V_0
+\int_0^t(1-e^{-\gamma^2(t-s)})  F(X_s^\gamma,\rho_s^\gamma ) ds+\sqrt{2}\int_0^t(1-e^{-\gamma^2(t-s)})dB_s
\,.
\end{align}
Then $\XG_t$ has the law $\rom_t$ for each $t\geq 0$. Denote by $\mc{C}([0,T];\RR^d)$ the space of all $\RR^d$-valued continuous functions on $[0,T]$ equipped with the usual uniform norm. Each continuous stochastic process $\XG$ may be seen as a $\mc{C}([0,T];\RR^d)$-valued random function and it induces a probability measure (or law, denoted by $\rom$) on $\mc{C}([0,T];\RR^d)$. We shall use the weak convergence in the space of probability measures on $\mc{C}([0,T];\RR^d)$.
We first use Aldous's tightness criteria to prove the tightness of the process $\{\XG\}_{\gamma>0}$. This means that  there exist a convergent subsequence of $\{\XG\}_{\gamma>0}$, which will be still denoted by $\{\XG\}_{\gamma>0}$, such that for some process $\widehat X$ it holds
$\XG \rightharpoonup \widehat X$  in the sense of distribution.  Next we verify that the limit process $\widehat X$ indeed satisfies \eqref{1MVeq} which is the underlying Mckean process of \eqref{adeq}. Furthermore we can obtain the quantified convergence rate by  comparing \eqref{onlyX} and \eqref{1MVeq} directly, which reads
\begin{eqnarray}
\sup_{t\in[0,T]}W_2^2(\rom_t,\rho_t)\leq \sup_{t\in[0,T]}\EE[|\XG_t-X_t|^2]\leq C\frac{1}{\gamma^2}\,.
\end{eqnarray}
 See Theorem \ref{thmlimit}.

For readers' convenience, we give a brief introduction of the Wasserstein metric in the following definition, we refer to \cite{ambrosio2008gradient} for more details.
	Let $1\leq p < \infty$ and $\mc{P}_p(\RR^{d})$ be the space of Borel probability measures on $\RR^{d}$ with finite $p$-moment. We equip this space with the Wasserstein distance 
	\begin{equation}\label{wassdis}
	W_p^{p}(\mu, \nu):=\inf\left\{\int_{\RR^d\times \RR^d} |z-\hat{z}|^{p}\ d\pi(\mu, \nu)\ \big| \ \pi \in \Pi(\mu, \nu)\right\}
	\end{equation}
	where $\Pi(\mu, \nu)$ denotes the collection of all Borel probability measures on $\RR^d\times \RR^d$ with marginals $\mu$ and $\nu$ in the first and second component respectively. The Wasserstein distance can also be expressed as
	\begin{equation}
	W_p^{p}(\mu, \nu) = \inf \left\{\mathbb{E}[|Z-\overline {Z}|^{p}]\right\}\,,
	\end{equation}
	where the infimum is taken over all joint distributions of the random variables $Z$, $\overline {Z}$ with marginals $\mu$, $\nu$ respectively.
Thanks to the Kantorovich duality, in the space $\mc{P}_1(\RR^d)$, we shall also use the following alternative representation
\begin{equation}\label{Kanto}
W_1(\mu,\nu)=\sup\left\{\big|\int_{\RR^d}\phi(x)\rd (\mu-\nu)(x)\big|:~\phi\in \mbox{Lip}(\RR^d),~ \|\phi\|_{\text {Lip }} \leq 1\right\}\,,
\end{equation}
where $\mbox{Lip}(\RR^d)$ is the space of Lipschitz continuous functions on $\RR^d$ and $\|\phi\|_{\text {Lip }}:=\sup _{x \neq y} \frac{|\phi(x)-\phi(y)|}{|x-y|}$.

Throughout this paper we assume the external potential function $\Phi$ satisfies
\begin{assum}\label{asum}
	The external potential function $0\leq\Phi\in\mbox{Lip}_{loc}(\RR^d)$
	\begin{itemize}
		\item[1.]  There exists some constant $C_{\Phi}>0$ such that
		\begin{equation}\label{Phibound}
		|\nabla\Phi(x)|\leq C_{\Phi}(1+|x|)\quad \mbox{ for all }x\in\RR^d,\quad \mbox{ and }  \norm{\nabla \Phi}_{\mbox{Lip}}\leq C_{\Phi}\,;
		\end{equation}
		\item[2.]	For any $r\in[0,\infty)$: $C_{\Phi,r}:=\sup_{x\in\RR^d}|\nabla\Phi(x)|^re^{-\Phi(x)}<\infty$\,.
	\end{itemize}
\end{assum}
Note that the above assumption allows us to consider both bounded and unbounded external potentials with at most quadratic growth at infinity, in particular the case $\Phi(x)=|x|^2/2$.

As for the interaction potential $K$, we will consider both regular and singular cases. We start with assuming $\nabla K\in L^\infty(\RR^d)\cap \mbox{Lip}(\RR^d)$ in Section 2. Then in Section 3, we assume that the potential $K\in L_{loc}^q(\RR^d)$ is singular, which in particular includes the case $|K(x)|\leq \frac{C}{|x|^\alpha}$ with $1\leq\alpha<d-2$, and the Newtonian potential case $\nabla K(x)=\pm\frac{x}{|x|^d}$. In the singular case we need higher-order regularity of solutions $\fm$. Firstly we need $\fm$ to be more than just a measure but a density function. Secondly, we require some uniform in $\gamma$ estimates of $\fm$ in weighted Sobolev space $W_{x,H}^{k,p}$, which will be collected from \cite{choi2020quantified}.

\section{Regular  interaction potential}
In this section we consider globally Lipschitz continuous and bounded interaction forces by
assuming $\nabla K\in L^\infty(\RR^d)\cap \mbox{Lip}(\RR^d)$. Denote by 
$$C_K:=\norm{\nabla K}_{L^\infty}+\norm{\nabla K}_{\mbox{Lip}}<\infty \,.$$
Then one can easily verify that
\begin{enumerate}
	\item For any $\mu\in\mc{P}(\RR^d)$, it holds
	\begin{equation}\label{infbound}
	\norm{\nabla K\ast \mu}_{L^\infty}\leq \norm{\nabla K}_{L^\infty}\leq C_K\,;
	\end{equation}
	\item For any $\mu,\nu\in\mc{P}(\RR^d)$ and any $x,y\in\RR^d$, it holds
	\begin{equation}\label{lipbound}
	|F(x,\mu)-F(y,\nu)|\leq C_F(|x-y|+W_2(\mu,\nu))
	\end{equation}
	with $C_F:=\norm{\nabla K}_{\mbox{Lip}}+\norm{\nabla \Phi}_{\mbox{Lip}}$.
\end{enumerate}

Under the above regular assumptions on $K$, standard result, see for example \cite{sznitman1991topics}, gives the well-posedness of the Mckean processes  \eqref{2MVeq} and \eqref{1MVeq}
\begin{thm}\label{thm-huang2021note}
	Let $\nabla K\in L^\infty(\RR^d)\cap \operatorname{Lip}(\RR^d)$.  For each $T>0$, there hold the following assertions.
	
	(i) If $(X_0,V_0)$ is distributed according to $f_0\in\mc{P}_2(\RR^{2d})$, then for each $\gamma>0$, the nonlinear SDE \eqref{2MVeq} admits a unique solution up to time $T$ with the initial data $(X_0,V_0)$ and it holds further that
	\begin{equation}\label{secmen}
	\sup\limits_{t\in[0,T]}\EE\left[|\XG_t|^2+|\VG_t|^2\right]\leq C\EE\left[ |X_0|^2+|V_0|^2\right] \,,
	\end{equation}
	where $C$ depends only on $C_K,C_\Phi,\gamma$ and $T$.
	
	(ii) If $X_0$ is distributed according to $\rho_0\in \mc{P}_2(\RR^{d})$, then SDE \eqref{1MVeq} admits a unique solution up to time $T$ with the initial data $X_0$ and it holds further that
	\begin{equation}\label{secmen-CBO}
	\sup\limits_{t\in[0,T]}\EE\left[|X_t|^2\right]\leq C\EE\left[ |X_0|^2 \right]\,,
	\end{equation}
	where $C$ depends only on $C_K,C_\Phi$ and $T$.
\end{thm}

The proof of the overdamped limit will proceed in two steps:
\begin{itemize}
	\item We prove a tightness result for the sequence of probability distributions $\{\rom\}_{\gamma>0}$ of $\{\XG\}_{\gamma>0}$ by using Aldous's tightness criteria.
	\item We will check that all the limit points of $\{\XG\}_{\gamma>0}$ as $\gamma\to \infty$ satisfy the Mckean process \eqref{1MVeq} underlying the aggregation-diffusion equation \eqref{adeq}.
\end{itemize}
For the sake of completeness, let us recall a result from the Aldous criteria \cite[Theorem 4.5]{jacod2002limit}.
\begin{lem}\label{lemAldous}
Let $\{X^n\}_{n\in \NN}$ be a sequence of random variables defined on a probability space $(\Omega,\mc{F},\PP)$ and valued in $\mc{C}([0,T];\RR^d)$. The sequence of probability distributions $\{\mu_{X^n}\}_{n\in \NN}$  of $\{X^n\}_{n\in \NN}$ is tight on $\mc{C}([0,T];\RR^d)$ if the following hold.

$(Con 1)$ For all $t\geq 0$, $\{\mu_{X_t^n}\}_{n\in\NN}$ the set of distributions of $\{X_t^n\}_{n\in \NN}$  is tight in $\RR^d$.

$(Con 2)$ For all $\varepsilon>0$, $\eta>0$, there exists $\delta_0>0$ and $n_0\in\NN$ such that for all $n\geq n_0$ and for all discrete-valued $\sigma(X^n_s;s\in[0,T])$-stopping times $\beta$ such that $0\leq \beta+\delta_0\leq T$,
\begin{equation}
\sup_{\delta\in[0,\delta_0]}\PP\left(|X^n_{\beta+\delta}-X^n_{\beta}|\geq \eta\right)\leq \varepsilon\,.
\end{equation}
 \end{lem}
\begin{thm}[Tightness]\label{thmtight}
For any $\gamma>0$ and $T>0$,	let $(\XG_t,\VG_t)_{t\in[0,T]}$ satisfy the system \eqref{2MVeq} up to time $T$ with $\nabla K\in L^\infty(\RR^d)\cap \operatorname{Lip}(\RR^d)$,  $\Psi$ satisfying Assumption \ref{asum},  and the initial data $\operatorname{Law}(X_0,V_0)=f_0\in \mc{P}_2(\RR^{2d})$. Then for each countable subsequence $\{\gamma_k\}_{k\in \NN}$ with $\lim_{k\rightarrow \infty} \gamma_k=\infty$, the sequence of probability distributions $\{\rho^{\gamma_k}\}_{k\in\NN}$ of $\{X^{\gamma_k}\}_{k\in \NN}$  is tight on $\mc{C}([0,T];\RR^d)$.
\end{thm}
\begin{proof}
	We apply the Aldous criteria in Lemma \ref{lemAldous} to the system $\{\XG\}_{\gamma>0}$ by verifying conditions $(Con 1)$ and $(Con 2)$.
	
	$\bullet$ \textit{Step 1: Checking $(Con 1)$. }  
	Let us recall \eqref{onlyX} and use Fubini's theorem (see \cite[Theorem 4.33]{da2014stochastic} for the stochastic version), then we have
		\begin{align}\label{Fubini}
	X_t^\gamma&=X_0+\gamma\int_0^t  e^{-\gamma^2\tau}V_0d\tau
	+\gamma^2\int_0^t\int_0^\tau e^{-\gamma^2(\tau-s)} F(X_s^\gamma,\rho_s^\gamma ) dsd\tau +\sqrt{2}\gamma^2\int_0^t\int_0^\tau e^{-\gamma^2(\tau-s)}dB_s d\tau \notag\\
	&=X_0+\gamma\int_0^t  e^{-\gamma^2\tau}V_0d\tau
	+\gamma^2\int_0^t\int_s^t e^{-\gamma^2(\tau-s)}d\tau  F(X_s^\gamma,\rho_s^\gamma ) ds+\sqrt{2}\gamma^2\int_0^t\int_s^t e^{-\gamma^2(\tau-s)}d\tau  dB_s \notag\\
	&=X_0+\frac{1}{\gamma}(1-e^{-\gamma^2t})V_0
	+\int_0^t(1-e^{-\gamma^2(t-s)})  F(X_s^\gamma,\rho_s^\gamma ) ds+\sqrt{2}\int_0^t(1-e^{-\gamma^2(t-s)})dB_s\,.
	\end{align}
	For $\gamma\geq 1$, it follows from H\"{o}lder's inequality that
	\begin{equation*}
	|X_t^\gamma|^2\leq 4|X_0|^2+4|V_0|^2+2T\int_0^t|F(X_s^\gamma,\rho_s^\gamma )|^2ds+8\left|\int_0^{t}(1-e^{-\gamma^2(t-s)})  dB_s\right|^2\,.
	\end{equation*}
	Here we have used the fact that for any sequence $\{a_i\}_{i=1}^n\geq 0$, one has
	\begin{equation}
	(\sum_{i=1}^{n}a_i)^2\leq n\sum_{i=1}^{n}a_i^2\,.
	\end{equation}
Using It\^{o}'s isometry  yields that
	\begin{align*}
	&\EE\left[\left|\int_0^{t}(1-e^{-\gamma^2(t-s)}) dB_s\right|^2\right]=\EE\left[\int_0^{t}|(1-e^{-\gamma^2(t-s)})|^2ds\right]\leq T
	\,.
	\end{align*}
Thus we have 
	\begin{align}
	\EE[|X_t^\gamma|^2]\leq 4\EE[|X_0|^2]+4\EE[|V_0|^2]+2T\int_0^t\EE[|F(X_s^\gamma,\rho_s^\gamma )|^2]ds+8T\,.
	\end{align}
It follows from \eqref{Phibound} and \eqref{infbound} that
\begin{align}\label{EFbound}
\EE[|F(X_s^\gamma,\rho_s^\gamma )|^2]&=\EE[|\nabla\Phi(X_s^\gamma)+\nabla K\ast \rom_s(X_s^\gamma)|^2]   \leq 2\EE[|\nabla\Phi(X_s^\gamma)|^2]+2C_K^2\notag\\
&\leq 2\EE[|C_\Phi(1+|X_s^\gamma|)|^2]+2C_K^2\leq 4C_\Phi^2+2C_K^2+4C_\Phi^2\EE[|X_s^\gamma|^2]
 \,,
\end{align}	
which leads to 
\begin{align}
\EE[|X_t^\gamma|^2]\leq 4\EE[|X_0|^2]+4\EE[|V_0|^2]+8C_\Phi^2T\int_0^t\EE[|X_s^\gamma|^2]ds+2T(4C_\Phi^2+2C_K^2)+8T\,.
\end{align}
Using Gronwall's inequality leads to
\begin{equation}\label{esOX2}
\EE[|X_t^\gamma|^2]
\leq \left(4\EE[|X_0|^2]+4\EE[|V_0|^2]+2T(4C_\Phi^2+2C_K^2)+8T\right)\exp\left(8C_\Phi^2T^2\right),\quad t\in[0,T]\,.
\end{equation}
This  implies that
\begin{equation} \label{uniform-bd}
\sup_{t\in[0,T]}\EE[|X_t^\gamma|^2] \leq C(\EE[|X_0|^2],\EE[|V_0|^2],C_\Phi,C_K,T)=:C_1
\end{equation}
where $C_1>0$ is a  constant independent of $\gamma$. So for any $\varepsilon>0$, there exists a compact subset $K_\varepsilon:=\{x:~|x|^2\leq \frac{C_1}{\varepsilon}\}$ such that by Markov's inequality
\begin{equation}
\rho_t^\gamma((K_\varepsilon)^c)=\PP(|X_t^\gamma|^2> \frac{C_1}{\varepsilon})\leq \frac{\varepsilon\EE[|X_t^\gamma|^2]}{C_1}\leq\varepsilon,\quad \forall ~\gamma\geq1\,.
\end{equation}
This means that for all $t\in[0,T]$, each countable subset of $\{\rho_t^\gamma\}_{\gamma\geq1}$ is tight in $\RR^d$, which verifies  condition $(Con 1)$ in Lemma \ref{lemAldous}.

$\bullet$ \textit{Step 2: Checking $(Con 2)$. }   Let $\beta$ be a $\sigma(X^\gamma_s;s\in[0,T])$-stopping time with discrete values such that $\beta+\delta_0\leq T$. 
Let us recall \eqref{onlyX} and compute
{\small 	\begin{align}\label{diff}
&X_{\beta+\delta}^\gamma-X_{\beta}^\gamma=\int_\beta^{\beta+\delta}V_\tau d\tau\notag\\
=&\gamma\int_\beta^{\beta+\delta}  e^{-\gamma^2\tau}V_0d\tau
+\gamma^2\int_\beta^{\beta+\delta}\int_0^\tau e^{-\gamma^2(\tau-s)} F(X_s^\gamma,\rho_s^\gamma ) dsd\tau +\sqrt{2}\gamma^2\int_\beta^{\beta+\delta}\int_0^\tau e^{-\gamma^2(\tau-s)}dB_s d\tau \notag\\
=&\gamma\int_\beta^{\beta+\delta}  e^{-\gamma^2\tau}V_0d\tau
+\gamma^2\int_0^\beta\int_\beta^{\beta+\delta} e^{-\gamma^2(\tau-s)}d\tau  F(X_s^\gamma,\rho_s^\gamma ) ds+\gamma^2\int_\beta^{\beta+\delta}\int_s^{\beta+\delta} e^{-\gamma^2(\tau-s)}d\tau  F(X_s^\gamma,\rho_s^\gamma ) ds\notag\\
& \quad +\sqrt{2}\gamma^2\int_0^\beta\int_\beta^{\beta+\delta} e^{-\gamma^2(\tau-s)}d\tau  dB_s +\sqrt{2}\gamma^2\int_\beta^{\beta+\delta}\int_s^{\beta+\delta} e^{-\gamma^2(\tau-s)}d\tau dB_s \notag\\
=&\frac{1}{\gamma}(e^{-\gamma^2\beta}-e^{-\gamma^2(\beta+\delta)})V_0 
+\int_0^\beta (e^{-\gamma^2(\beta-s)}-e^{-\gamma^2(\beta+\delta-s)})  F(X_s^\gamma,\rho_s^\gamma ) ds
+\int_\beta^{\beta+\delta} (1-e^{-\gamma^2(\beta+\delta-s)})  F(X_s^\gamma,\rho_s^\gamma ) ds\notag\\
& \quad +\sqrt{2}(e^{-\gamma^2\beta}-e^{-\gamma^2(\beta+\delta)})\int_0^\beta e^{\gamma^2s} dB_s+\sqrt{2}\int_\beta^{\beta+\delta} (1-e^{-\gamma^2(\beta+\delta-s)})dB_s\,.
\end{align}}
Note here that the multiplier $(e^{-\gamma^2\beta}-e^{-\gamma^2(\beta+\delta)})$ cannot enter the stochastic integral due to the non-anticipativity required for It\^{o} integrals and associated moment estimates.

Notice that it holds $|e^{-x}-e^{-y}|\leq |x-y|\wedge 1 $ for all $x,y\in[0,\infty)$ and $\xi\in[0,1]$. Then it is easy to compute that for each $q\geq 1$, $\xi\in[0,1]$ and $\tau\in[0,T]$,
\begin{align}\label{est-11} 
&\int_0^{\tau}
\left| e^{-\gamma^2(\tau-s)} - e^{-\gamma^2(\tau+\delta-s)}   \right|^q \,ds
\leq
\int_0^{\tau}
\left( e^{-\gamma^2(\tau-s)} - e^{-\gamma^2(\tau+\delta-s)}  \right) \,ds\nn\\
=&
\frac{1}{\gamma^2} \left( 1- e^{-\gamma^2\delta}  \right)
-\frac{1}{\gamma^2} \left( e^{-\gamma^2\tau} -e^{-\gamma^2(\tau+\delta)} \right)
\leq \frac{1}{\gamma^2} \cdot (\gamma^2\delta)^\xi=(\frac{1}{\gamma^2})^{1-\xi}\delta^\xi,
\end{align}
and in particular it holds
$$
\int_{\beta}^{\beta+\delta} \left(1-e^{-\gamma^2(\beta+\delta-s)}\right)^q ds \leq \int_{\beta}^{\beta+\delta} 1 \,ds =\delta.
$$

It is easy to see that
\begin{align}
\EE\left[\frac{1}{\gamma^2}|(e^{-\gamma^2\beta}-e^{-\gamma^2(\beta+\delta)})V_0|^2\right]\leq\frac{1}{\gamma^2} \gamma^2\delta\EE[|e^{-\gamma^2\beta}-e^{-\gamma^2(\beta+\delta)}||V_0|^2]\leq \delta \EE[|X_0|^2] \,.
\end{align}
Moreover, we notice that
\begin{align}
&\EE\left[|\int_0^\beta (e^{-\gamma^2(\beta-s)}-e^{-\gamma^2(\beta+\delta-s)})  F(X_s^\gamma,\rho_s^\gamma ) ds|^2\right]\notag\\
\leq &  \EE\left[\int_0^\beta |e^{-\gamma^2(\beta-s)}-e^{-\gamma^2(\beta+\delta-s)}|^2ds\int_0^\beta |F(X_s^\gamma,\rho_s^\gamma ) |^2ds\right]
\leq \delta \int_0^T  \EE[|F(X_s^\gamma,\rho_s^\gamma ) |^2]ds\notag\,,
\end{align}
and 
\begin{align*}
&\EE\left[|\int_\beta^{\beta+\delta} (1-e^{-\gamma^2(\beta+\delta-s)})  F(X_s^\gamma,\rho_s^\gamma ) ds|^2\right]
\leq \delta \EE\left[\int_\beta^{\beta+\delta} |F(X_s^\gamma,\rho_s^\gamma ) |^2ds\right]\leq \delta\int_0^{T}\EE\left[ |F(X_s^\gamma,\rho_s^\gamma ) |^2\right]ds\,.
\end{align*}
Applying It\^{o}'s isometry  one has
\begin{align}
&\EE\left[|\int_\beta^{\beta+\delta} (1-e^{-\gamma^2(\beta+\delta-s)}) dB_s|^2\right]
= \EE\left[\int_\beta^{\beta+\delta} |1-e^{-\gamma^2(\beta+\delta-s)}|^2ds\right]\leq \delta\,.
\end{align}

Particularly, let us look at 
$$
Z^{\gamma,\delta}_{t}:
= ( e^{-\gamma^2t} -e^{-\gamma^2(t+\delta)} ) 
\int_0^{t} e^{\gamma^2s}dB_s, \quad t\in [0,T],
$$
and try to derive an estimate on $Z^{\gamma,\delta}_{\beta}$. Basic calculations as above yield that
\begin{align}\label{est-111}
\EE\left[
\int_0^T \Big|  Z^{\gamma,\delta}_{t} \Big|^2 dt
\right]
&
=
\EE\left[
\int_0^T \Big|   
\int_0^{t}  
( e^{-\gamma^2t} -e^{-\gamma^2(t+\delta)} ) 
e^{\gamma^2s}dB_s \Big|^2 dt
\right]
\nonumber\\
&=\int_0^T\int_0^t |( e^{-\gamma^2t} -e^{-\gamma^2(t+\delta)} ) 
e^{\gamma^2s}|^2dsdt\leq C\delta
\end{align}
where we have used  It\^{o}'s isometry and the estimate \eqref{est-11} with $\zeta=1$ and the constant $C$ is independent of $\gamma$ and $\delta$. Thus, the process
$$
M^{\gamma,\delta}_t:= \int_0^t  (Z^{\gamma,\delta}_s)' dB_s, \quad t\in[0,T],
$$
is a integrable continuous martingale; indeed, Doob's martingale inequality gives
\begin{align}
\EE\left[ \max_{t\in[0,T]} \left| M^{\gamma,\delta}_t \right|\right]
\leq C\EE\left[ \int_0^T  \left| (Z^{\gamma,\delta}_s)'  \right|ds \right]
\leq 
C \left(
\EE\left[ \int_0^T  \left|  Z^{\gamma,\delta}_s \right|^2 ds \right] 
 \right)^{1/2} \leq
C  \delta^{1/2}, \label{est-max}
\end{align}
with $C$ being independent of $\gamma$ and $\delta$. On the other hand, it is easy to see that $Z^{\gamma,\delta}$ satisfies the following SDE
$$
dZ^{\gamma,\delta}_t= -\gamma^2Z^{\gamma,\delta}_tdt + (1-e^{-\gamma^2\delta}) \,dB_t, \quad t>0;\quad Z^{\gamma,\delta}_0=0.
$$
By It\^o-Doeblin formula, it holds that
for all $t\in[0,T]$,
\begin{align}
|Z^{\gamma,\delta}_t|^2
&=\int_0^t \Big| ( e^{-\gamma^2 (t-s)} -e^{-\gamma^2 (t+\delta-s)} ) \Big|^2ds 
+ 2 \int_0^t e^{-2\gamma^2 (t-s)}(  1 -e^{-\gamma^2 \delta}) d M^{\gamma,\delta}_s
\nonumber\\
&\leq \delta
+2 (  1 -e^{-\gamma^2 \delta}) M^{\gamma,\delta}_t 
- 4 \int_0^t \gamma^2e^{-2\gamma^2 (t-s)}(  1 -e^{-\gamma^2 \delta})  M^{\gamma,\delta}_s \,ds
\nonumber\\
&\leq
\delta
+2\left| M^{\gamma,\delta}_t \right|
+ 4\gamma^2  \max_{s\in[0,T]} \left| M^{\gamma,\delta}_s \right| 
\int_0^t e^{-2\gamma^2 (t-s)}  \,ds
\leq  \delta
+4 \max_{s\in[0,T]} \left| M^{\gamma,\delta}_s \right|  , \quad\text{a.s.,}\label{est-222}
\end{align}
where the integration by parts formula is applied to the stochastic integral in the first line and in the second inequality, we used estimate \eqref{est-11} with $\zeta$ equal to $ 1$. Combined with \eqref{est-max}, it yields that
\begin{align*}
\EE\left[ \Big|Z^{\gamma,\delta}_{\beta} \Big|^2\right]
\leq  \EE\left[ \max_{t\in[0,T]}\Big|Z^{\gamma,\delta}_{t} \Big|^2\right]
\leq C (\delta^{1/2}+\delta),
\end{align*}
where the constant $C$ is independent of $\beta,\gamma,$ and $\delta$.

Thus we have
\begin{align}
&\EE[|\XG_{\beta+\delta}-\XG_{\beta}|^2] 
\leq 5\delta \EE[|X_0|^2]+10\delta \int_0^T  \EE[|F(X_s^\gamma,\rho_s^\gamma ) |^2]ds +10\delta+10C (\delta^{1/2}+\delta)\notag\\
\leq&5\delta \EE[|X_0|^2]
+10\delta\left( 4C_\Phi^2T+2C_K^2T+4C_\Phi^2\int_0^T\EE[|X_s^\gamma|^2] ds\right)+10\delta+10C (\delta^{1/2}+\delta)  \,,
\end{align}
where we have used \eqref{EFbound} in the last inequality.
Recalling the fact that for all $\gamma\geq 1$
$$\sup_{s\in[0,T]}\EE[|X_s^\gamma|^2]\leq C(\EE[|X_0|^2],\EE[|V_0|^2],C_{\Phi},C_K,T)\,,$$
it easy to see that
 \begin{align*}
\EE[|\XG_{\beta+\delta}-\XG_{\beta}|^2]\leq C(\EE[|X_0|^2],\EE[|V_0|^2],C_{\Phi},C_K,T)(\delta+\delta^{\frac{1}{2}})\,.
\end{align*}
Hence for any $\varepsilon>0$, $\eta>0$, there exists some $\delta_0$ and $\gamma_0=1$, such that for all $\gamma\geq 1$ it holds that
\begin{equation}
\sup_{\delta\in[0,\delta_0]}\PP(|\XG_{\beta+\delta}-\XG_{\beta}|^2\geq \eta)\leq \sup_{\delta\in[0,\delta_0]}\frac{\EE[|\XG_{\beta+\delta}-\XG_{\beta}|^2]}{\eta}\leq \varepsilon\,.
\end{equation}
This completes the verification of condition $Con 2$ in Lemma \ref{lemAldous}. 
	\end{proof}

 Next we shall identify the limit process.
\begin{thm}[Overdamped limit]\label{thmlimit} 
For any $\gamma>0$ and $T>0$,	let $(\XG_t,\VG_t)_{t\in[0,T]}$ satisfy the system \eqref{2MVeq} up to time $T$ with $\nabla K\in L^\infty(\RR^d)\cap \mbox{Lip}(\RR^d)$,  $\Psi$ satisfying Assumption \ref{asum},  and the initial data $\operatorname{Law}(X_0,V_0)=f_0\in \mc{P}_2(\RR^{2d})$. Then as $\gamma\rightarrow \infty$, the sequence of  stochastic processes $\{\XG\}_{\gamma>0}$  converge weakly to $ X$, which is the unique solution to the following SDE:
	\begin{align}\label{1stSDE}
X_t=X_0+\int_0^tF(X_s,\rho_s)ds +\sqrt{2}B_t\,.
\end{align}
Moreover it holds that
\begin{equation}\label{est-convegence-m}
\sup_{t\in[0,T]}\EE[|\XG_t-X_t|^2]\leq \frac{C}{\gamma^2}e^{CT}\,,
\end{equation}
where $C$ depends only on $\EE[|X_0|^2],\EE[|V_0|^2],C_{\Phi},C_K,C_F$ and $T$.
\end{thm}
\begin{rmk}
	It follows from the definition of Wasserstein distance that
	\begin{equation}
	\sup_{t\in[0,T]}W_2^2(\rho^\gamma_t,\rho_t)\leq  \sup_{t\in[0,T]}\EE[|\XG_t-X_t|^2]\leq \frac{C}{\gamma^2}e^{CT}\,,
	\end{equation}
	which  is consistent with the result obtained in \cite[Theorem 1.3]{choi2020quantified}.
\end{rmk}

\begin{proof}
By Theorem \ref{thmtight} each subsequence $\{X^{\gamma_k}\}_{k\in\mathbb N}$ with $\gamma_0\geq 1$ and $\gamma_k$ converging increasingly to $0$ as $k\rightarrow \infty$  admits a subsequence (denoted w.l.o.g. by itself) that converges weakly.
This means that there exists some process $\widehat X$ as random variables valued in $\mc{C}([0,T];\RR^{d})$ such that
\begin{equation}\label{converge-as}
X\rightharpoonup \widehat X
\end{equation}
in the sense of distribution.

Recall the SDE satisfied by $X^{\gamma_k}$ in \eqref{Fubini}
\begin{align}\label{eq-X-mk}
X_t^{\gamma_k}=X_0+\frac{1}{\gamma_k}(1-e^{-\gamma_k^2t})V_0
+\int_0^t(1-e^{-\gamma_k^2(t-s)})  F(X_s^{\gamma_k},\rho_s^{\gamma_k }) ds+\sqrt{2}\int_0^t(1-e^{-\gamma_k^2(t-s)})dB_s\,. 
\end{align}
Notice that the estimate in \eqref{uniform-bd} implies that
\begin{equation}\label{estimate-L4-bd}  
\sup_{k\in\mathbb N} \sup_{t\in[0,T]}\EE[|X_t^{\gamma_k}|^2] \leq  C_1, \quad \text{ and thus, }\quad
\sup_{t\in[0,T]} \EE[|\widehat X_t|^2] \leq  C_1,
\end{equation}
with the constant $C_1$ being independent of $\gamma_k$.
We recall
$$
\EE[|F(X_s^\gamma,\rho_s^\gamma )|^2]\leq 4C_\Phi^2+2C_K^2+4C_\Phi^2\EE[|X_s^\gamma|^2]\,.
$$
Denoting by $ \rho_t$ the probability distribution of $X_t$ for $t\in[0,T]$, thus we have
\begin{equation}\label{supFbound}
\sup_{k\in\mathbb N} \sup_{t\in[0,T]}|F(X_s^{\gamma_k},\rho_s^{\gamma_k} )|\leq  C(C_1,C_K,C_\Phi), \quad \text{ and }\quad
\sup_{t\in[0,T]} |F( X_s,{\rho}_s)|\leq  C(C_1,C_K,C_\Phi)\,.
\end{equation}
 
Furthermore, by \eqref{lipbound} and the definition of Wasserstein distance, one has
\begin{equation}
 \EE[	|F(X_s^{\gamma_k},\rho^{\gamma_k}_s)- F( X_s, \rho_s)|^2]\leq C_F^2\EE[||X_s^{\gamma_k}- X_s|+W_2(\rho^{\gamma_k}_s, \rho_s)|^2]\leq 2 C_F^2\EE[|X_s^\gamma- X_s|^2],
\end{equation}
and thus,
\begin{align}\label{P1}
&\EE \left[ \left| \int_0^t(1-e^{-\gamma_k^2 (t-s)}) F(X_s^{\gamma_k},\rho^{\gamma_k}_s)ds- 
\int_0^t  F(X_s,\rho_s)ds \right|^2\right]\\                     
\leq &
2 \EE \left[  \left|
 \int_0^t(1-e^{-\gamma_k^2 (t-s)}) (F(X_s^{\gamma_k},\rho^{\gamma_k}_s)- F( X_s, \rho_s))ds \right|^2 \right]  +2\EE\left[\left|\int_0^te^{-\gamma_k^2 (t-s)}F(X_s,\rho_s)ds\right|^2\right]\\
 \leq&  C \EE \left[ \int_0^t  \left|  X_s-X_s^{\gamma_k} \right|^2 ds \right]+C\int_0^te^{-2\gamma_k^2 (t-s)}ds\EE\left[\int_0^T|F( X_s,\rho_s)|^2ds\right]  \\
\leq&
C \EE \left[ \int_0^t  \left|   X_s-X_s^{\gamma_k} \right|^2 ds \right]
+ C \frac{1}{\gamma_k^2}\,,
\end{align}
where the constant $C$ is independent of $k$, and we have used  the boundedness in \eqref{supFbound}.
For the stochastic integrals, it holds analogously that
\begin{align}\label{P2}
&\EE \left[ \left|
 \int_0^t(1-e^{-\gamma_k^2 (t-s)}) dB_s
-  \int_0^t dB_s \right|^2\right]=\EE \left[ \left|
\int_0^te^{-\gamma_k^2 (t-s)} dB_s \right|^2\right]
\nonumber \\
=&\EE \left[
\int_0^te^{-2\gamma_k^2 (t-s)} ds \right]\leq C\frac{1}{\gamma_k^2}\,.
\end{align}
Additionally, it is obvious that
\begin{equation}\label{P3}
|\frac{1}{\gamma_k}(1-e^{-\gamma_k^2t})V_0|^2\leq \frac{1}{\gamma_k^2}|V_0|^2\,.
\end{equation}

Therefore collecting estimates \eqref{P1}--\eqref{P3} and subtracting both sides of SDEs \eqref{eq-X-mk} and  \eqref{1stSDE}, one has
\begin{align}
\EE[|X_t^{\gamma_k}-X_t|^2]\leq C\int_0^t\EE[|X_s^{\gamma_k}-X_s|^2]ds+C\frac{1}{\gamma_k^2}\,,
\end{align}
where $C$ depends only on $\EE[|X_0|^2],\EE[|V_0|^2],C_{F},C_\Phi,C_K$ and $T$. By Gronwall's inequality implies that
\begin{equation}\label{convergence-l2}
\sup_{t\in[0,T]}\EE[|X_t^{\gamma_k}-X_t|^2]\leq C\frac{1}{\gamma_k^2}e^{CT}\,.
\end{equation}
In view of both the convergences \eqref{converge-as} and \eqref{convergence-l2}, we must have $\widehat X = X$. 
Finally, due to the arbitrariness of the subsequence $\{X^{\gamma_k}\}_{k\in\mathbb N}$ and the uniqueness of $X$, we conclude that as $\gamma\rightarrow \infty$, the sequence of  stochastic processes $\{\XG\}_{\gamma>0}$  converge weakly to  the unique solution $ X$ to SDE \eqref{1stSDE}, with the estimate \eqref{est-convegence-m} following in the same way as \eqref{convergence-l2}.
\end{proof}

\section{Singular interaction potential}
In this section we assume that $\nabla K\in L_{loc}^q(\RR^d)$ for some $q\in (1,\infty]$. More specifically it satisfies
\begin{assum}\label{asum1}
There exist some $C_K'>0$ such that
\begin{equation}
C_K':=\norm{\nabla K}_{L^q(B_{2R})}+\norm{\nabla K}_{W^{1,\infty}(\RR^d \backslash B_R)}<\infty \mbox{ for some }R>0\mbox{ and }q\in (1,\infty]\,,
\end{equation}
where $B_R$ represents a ball of radius $R$ and centered at origin.
\end{assum}

As it has been mentioned in Introduction, for singular case we require more regular solutions $\fm$ to the kinetic equation \eqref{VFPeq}. For $p\in[1,\infty)$, we consider the space of weighted measurable functions $L_H^p(\RR^{2d})$ with the norm
\begin{equation}
\norm{\fm_t}_{L_H^p}:=\left(\iint_{\RR^d\times\R^d}|\fm_t|^pe^{(p-1)H}dxdy\right)^{\frac{1}{p}}\,,
\end{equation}
where $H(x,v)=\Phi(x)+\frac{|v|^2}{2}$. For any integer $k\in \NN$, $W_{x,H}^{k,p}$ represents $L_H^p$ Sobolev space of $k$-th order in $x$ with the norm
\begin{equation}
\norm{\fm_t}_{W_{x,H}^{k,p}}:=\left(\sum_{|\alpha|\leq k}\iint_{\RR^d\times\R^d}|\nabla_x^\alpha\fm_t|^pe^{(p-1)H}dxdy\right)^{\frac{1}{p}}\,.
\end{equation}
First we recall the following implied regularity for $\rom$:
\begin{lem}{\cite[Lemma 2.9]{choi2020quantified}}\label{lmineq}
If $\fm_t\in W_{x,H}^{1,p}(\RR^{2d})$, then $\rom_t\in W^{1,p}(\RR^d)$. In particular, if $p>d$, one has $\rom_t\in L^\infty(\RR^d)$\,.
\end{lem}
Now we state the result of the well-posedness of local-in-time solutions to the VFP equation \eqref{VFPeq} and $\gamma$ independent estimates:
\begin{thm}{\cite[Theorem 1.4]{choi2020quantified}}\label{thmexist}
	Let $T>0$ and suppose that the interaction potential $K$ satisfy Assumption \ref{asum1}. Let the initial data $f_0\in \mc{P}_2\cap W_{x,H}^{1,p}(\RR^{2d})$, $p>\max\{d,q/(q-1)\}$. There exists a positive time $T_p\in (0,T]$, and a unique solution
	\begin{equation}
f^{\gamma}\in \mathcal{C}\left(\left[0, T_{p}\right] ; \mathcal{P}_{2}\left(\mathbb{R}^{d} \times \mathbb{R}^{d}\right)\right) \cap L^{\infty}\left(\left[0, T_{p}\right] ; W_{x, H}^{1, p}\left(\mathbb{R}^{d} \times \mathbb{R}^{d}\right)\right)
	\end{equation}
	to \eqref{VFPeq} satisfying
	\begin{equation}
	\sup_{\gamma\geq 1}\sup_{t\in[0,T_p]}\norm{\fm_t}_{W_{x, H}^{1, p}}<\infty\,.
	\end{equation}
Lemma \ref{lmineq}  implies that
	\begin{equation}\label{rhoes}
	\sup_{\gamma\geq 1}\sup_{t\in[0,T_p]}\norm{\rom_t}_{L^\infty}<\infty\,.
	\end{equation}
\end{thm}
 Using above theorem we can proved the uniform bound of $\nabla K\ast \rom_t$ under additional assumption of $K$:
 \begin{assum}\label{asum2}
 	The interaction potential $K$ satisfies one of the following conditions:
 	\begin{itemize}
 		\item $\nabla K\in W^{1,1}(B_{2R})$;
 		\item $K$ is given by the Newtonian potential, i.e. $\pm\Delta K=\delta_0$, where $\delta_0$ denotes the Dirac measure on $\RR^d$ giving unit mass to the origin.
 	\end{itemize}
 \end{assum}

\begin{lem}\label{lmuniK}
	Assume that $K$ satisfy Assumption \ref{asum1}--\ref{asum2} and let $\fm,\rom$ be the regular solution to \eqref{VFPeq} obtained in Theorem \ref{thmexist}, It holds that
	\begin{equation}\label{uniK}
\sup_{\gamma\geq 1}\sup_{t\in[0,T_p]}\norm{\nabla K\ast \rom_t}_{W^{1,\infty}}<\infty\,.
	\end{equation}
\end{lem}
 \begin{proof}
 	For $x\in \RR^d$ we obtain 
 	\begin{equation}
 	\nabla K\ast \rom_t(x)=\int_{\RR^d}\nabla K(x-y)\rom_t(y)dy\leq \int_{B_R}|\nabla K|(x-y)\rom_t(y)dy+C_K'\,.
 	\end{equation}
 	Notice that
 	\begin{equation}
 	\int_{B_R}|\nabla K|(x-y)\rom_t(y)dy\leq \norm{\nabla K}_{L^q(B_R)}\norm{\rom_t}_{L^{q/(q-1)}}<C(C_K',\norm{\rom_t}_{L^\infty})\,.
 	\end{equation}
 	By \eqref{rhoes}, this gives
 	\begin{equation}
 	\sup_{\gamma\geq 1}\sup_{t\in[0,T_p]}\norm{\nabla K\ast \rom_t}_{L^\infty}<\infty\,.
 	\end{equation}
 	If we additional assume that  $\nabla K\in W^{1,1}(B_{2R})$, following the same arguments as above we can easily obtain
 	\begin{equation}
 	\sup_{\gamma\geq 1}\sup_{t\in[0,T_p]}\norm{\nabla K\ast \rom_t}_{W^{1,\infty}}<\infty\,.
 	\end{equation}
 	On the other hand if we additional assume that $\pm\Delta K=\delta_0$, which means that $\nabla\cdot (\nabla K\ast \rom_t) =\pm\rom_t$. Thus we have $\norm{\nabla K\ast \rom_t}_{W^{1,\infty}}\leq C\norm{\rom_t}_{L^\infty}$, which implies \eqref{uniK}.
 	\end{proof}

Let us give a Lipschitz-type estimate for the interaction force $\nabla K$:
\begin{lem}\label{lmlipK}
	Assume that $K$ satisfy Assumption \ref{asum1}--\ref{asum2}. Let $\XG,X$ be two random variables with density $\rom,\rho \in\mc{P}_2\cap L^\infty(\RR^d)$ such that
	\begin{equation}
		C_\infty:=\max\{\|\rom\|_{L^\infty},\|\rho\|_{L^\infty}\}<\infty\,.
	\end{equation}
	Then it holds that
	\begin{equation}\label{case1}
	\EE[|\nabla K\ast\rom (\XG)-\nabla K\ast\rho (X)|^2 ]\leq C\left(\norm{\nabla K\ast\rom}_{W^{1,\infty}} ,C_\infty, \norm{\nabla^2 K}_{L^1(B_{2R})}^2,C_K'\right)\EE[|\XG-X|^2]
	\end{equation}
	or
		\begin{equation}\label{case2}
	\EE[|\nabla K\ast\rom (\XG)-\nabla K\ast\rho (X)|^2 ]\leq C\left(\norm{\nabla K\ast\rom}_{W^{1,\infty}} ,C_\infty\right)\EE[|\XG-X|^2]\,.
	\end{equation}
\end{lem}
\begin{proof}
	Let us split the error
	\begin{align}\label{split}
	&|\nabla K\ast\rom (\XG)-\nabla K\ast\rho (X)|^2=2|\nabla K\ast\rom (\XG)-\nabla K\ast\rom (X)|^2+2|\nabla K\ast(\rom-\rho) (X)|^2\notag \\
	\leq &2\norm{\nabla K\ast\rom}_{W^{1,\infty}}|\XG-X|^2+2|\nabla K\ast(\rom-\rho) (X)|^2\,.
	\end{align}
	The following proof will be divided into two cases:
	
	If $K$ satisfies Assumption \ref{asum1} and $\nabla K\in W^{1,1}(B_{2R})$, then 
	\begin{align}
	|\nabla K\ast(\rom-\rho) (X)|&\leq \left|\int_{(X-y)\in B_R}\nabla K(X-y)(\rom-\rho)(y)dy\right|+\left|\int_{(X-y)\in \RR^d\backslash B_R}\nabla K(X-y)(\rom-\rho)(y)dy\right|\notag\\
	&\leq \left|\int_{(X-y)\in B_R}\nabla K(X-y)(\rom-\rho)(y)dy\right|+\norm{\nabla K}_{\mbox{Lip}(\RR^d\backslash B_R)}W_1(\rom,\rho)\,,
	\end{align}
	where we have used the Kantorovich duality \eqref{Kanto} in the second inequality. 	Notice that according to the estimate $(2)$ in \cite[Theorem 3.3]{choi2020quantified} that
	\begin{align}
	&\EE\left[\left|\int_{(X-y)\in B_R}\nabla K(X-y)(\rom-\rho)(y)dy\right|^2\right]\leq C_\infty  \int_{\RR^d}\left|\int_{(x-y)\in B_R}\nabla K(x-y)(\rom-\rho)(y)dy\right|^2dx\notag\\
	\leq& C_\infty \norm{\nabla^2 K}_{L^1(B_{2R})}^2W_2^2(\rom,\rho)\,.
	\end{align}
	Thus we have
	\begin{align}
	\EE[|\nabla K\ast(\rom-\rho) (X)|^2]&\leq 2C_\infty \norm{\nabla^2 K}_{L^1(B_{2R})}^2W_2^2(\rom,\rho)+2 \norm{\nabla K}_{\mbox{Lip}(\RR^d\backslash B_R)}^2W_1^2(\rom,\rho)\notag\\
	&\leq \left(2C_\infty \norm{\nabla^2 K}_{L^1(B_{2R})}^2+2 C_K'^2\right)W_2^2(\rom,\rho)\notag\\
	&\leq \left(2C_\infty \norm{\nabla^2 K}_{L^1(B_{2R})}^2+2 C_K'^2\right)\EE[|\XG-X|^2]\,.
	\end{align}
	Then it follows from \eqref{split} that
	\begin{align}
	&\EE[|\nabla K\ast\rom (\XG)-\nabla K\ast\rho (X)|^2]\notag\\
	\leq &2\norm{\nabla K\ast\rom}_{W^{1,\infty}}\EE[|\XG-X|^2]+\left(4C_\infty \norm{\nabla^2 K}_{L^1(B_{2R})}^2+4 C_K'^2\right)\EE[|\XG-X|^2]\,,
	\end{align}
which proves \eqref{case1}.

If $K$ satisfies Assumption \ref{asum1} and it is given by Newtonian potential, according to $(3)$ in \cite[Theorem 3.3]{choi2020quantified}  one has
\begin{equation}
\norm{\nabla K\ast (\rom-\rho)}_{L^2}^2\leq C_\infty W_2^2(\rom,\rho)\leq C_\infty \EE[|\XG-X|^2]\,.
\end{equation}
Then taking expectation on both sides of \eqref{split} leads to
\begin{align}
&\EE[|\nabla K\ast\rom (\XG)-\nabla K\ast\rho (X)|^2]\leq 2\norm{\nabla K\ast\rom}_{W^{1,\infty}}\EE[|\XG-X|^2]+2\EE[|\nabla K\ast(\rom-\rho) (X)|^2]\notag\\
\leq& 2\norm{\nabla K\ast\rom}_{W^{1,\infty}}\EE[|\XG-X|^2]+2C_\infty \norm{\nabla K\ast (\rom-\rho)}_{L^2}^2\notag\\
\leq &\left(2\norm{\nabla K\ast\rom}_{W^{1,\infty}}+2C_\infty\right) \EE[|\XG-X|^2]\,.
\end{align}
This yields \eqref{case2}.
	\end{proof}
Now we can prove the boundness of the second moment uniformly in $\gamma$. In the following we shall use $C>0$ to denote a generic constant independent of $\gamma$ even though it is different from line to line.
\begin{thm}\label{thm2nd}
	Assume that $\Phi$ satisfy Assumption \ref{asum} and $K$ satisfy Assumption \ref{asum1}--\ref{asum2}. Let $(\XG_t,\VG_t)_{t\in[0,T_p]}$ be the unique solution to \eqref{2MVeq} up to time $T_p$ with the initial data $\mbox{Law}(X_0,V_0)=f_0\in \mc{P}_2\cap W_{x,H}^{1,p}(\RR^{2d})$, $p>\max\{d,q/(q-1)\}$.  It holds that
	\begin{equation}
	\sup_{\gamma\geq 1}\sup_{t\in[0,T_p]}\EE[|X_t^\gamma|^2]<\infty\,.
	\end{equation}
\end{thm}
\begin{proof}
	It follows from the same argument as in the Step 1 of Theorem \ref{thmtight} that
	\begin{align}
	\EE[|X_t^\gamma|^2]\leq C+C\int_0^t\EE[|F(X_s^\gamma,\rho_s^\gamma )|^2]ds\,,
	\end{align}
	where $C$ is independent of $\gamma$.
	It follows from \eqref{Phibound} and \eqref{uniK} that
	\begin{align}\label{EFbound1}
	\EE[|F(X_s^\gamma,\rho_s^\gamma )|^2]&=\EE[|\nabla\Phi(X_s^\gamma)+\nabla K\ast \rom_s(X_s^\gamma)|^2]   \leq 2\EE[|\nabla\Phi(X_s^\gamma)|^2]+C\notag\\
	&\leq 2\EE[|C_\Phi(1+|X_s^\gamma|)|^2]+C\leq C+C\EE[|X_s^\gamma|^2]
	\,,
	\end{align}	
	which leads to
		\begin{align}
	\EE[|X_t^\gamma|^2]\leq C+C\int_0^t\EE[|X_s^\gamma|^2]ds\,.
	\end{align}
	Applying Gronwall's inequality completes the proof.
	\end{proof}
Using the theorem above we obtain  the overdamped limit for the singular case:
\begin{thm}
Under the same assumptions as in Theorem \ref{thm2nd}. Let $(\XG_t,\VG_t)_{t\in[0,T_p]}$ be the unique solution to \eqref{2MVeq} up to time $T_p$, and $(X_t)_{t\in[0,T_p]}$ be the unique solution to \eqref{1MVeq}.
It holds that
\begin{equation}
\sup_{t\in[0,T_p]}\EE[|\XG_t-X_t|^2]\leq \frac{C}{\gamma^2}e^{CT}\,,
\end{equation}
where $C>0$ is independent of $\gamma\geq 1$.
\end{thm}
\begin{proof}
	It follows from the same argument as in Theorem \ref{thmlimit} that
	\begin{align}
	\EE[|\XG_t-X_t|^2]\leq C\int_0^t\EE[|F(\XG_s,\rom_s)-F(X_s,\rho_s)|^2]ds+C\frac{1}{\gamma^2}\,.
	\end{align}
	Applying assumption \eqref{Phibound} on $\Phi$, Lemma \ref{lmuniK} and Lemma \ref{lmlipK} one has
\begin{align}
\EE[|F(\XG_s,\rom_s)-F(X_s,\rho_s)|^2]&\leq 2\EE[|\nabla\Phi(\XG_s)-\nabla\Phi (X_s)|^2]+2\EE[|\nabla K\ast \rom_s(\XG_s)-\nabla K\ast \rho_s(X_s)|^2]\notag\\
&\leq C\EE[|\XG_s-X_s|^2]\,,
\end{align}
where $C$ is independent of $\gamma$. This implies
	\begin{align}
\EE[|\XG_t-X_t|^2]\leq C\int_0^t\EE[|\XG_s-X_s|^2]ds+C\frac{1}{\gamma^2}\,.
\end{align}
Applying Gronwall's inequality finishes the proof.
	\end{proof}

\bibliographystyle{amsxport}
\bibliography{overdamp}

\begin{bibdiv}
\begin{biblist}

\bib{ambrosio2008gradient}{book}{
      author={Ambrosio, Luigi},
      author={Gigli, Nicola},
      author={Savar{\'e}, Giuseppe},
       title={Gradient flows: In metric spaces and in the space of probability
  measures},
   publisher={Springer Science \& Business Media},
        date={2008},
}

\bib{bellomo2011modeling}{article}{
      author={Bellomo, Nicola},
      author={Dogbe, Christian},
       title={On the modeling of traffic and crowds: A survey of models,
  speculations, and perspectives},
        date={2011},
     journal={SIAM review},
      volume={53},
      number={3},
       pages={409\ndash 463},
}

\bib{bolley2011stochastic}{article}{
      author={Bolley, Fran{\c{c}}ois},
      author={Canizo, Jos{\'e}~A},
      author={Carrillo, Jos{\'e}~A},
       title={Stochastic mean-field limit: {non-Lipschitz} forces and
  swarming},
        date={2011},
     journal={Mathematical Models and Methods in Applied Sciences},
      volume={21},
      number={11},
       pages={2179\ndash 2210},
}

\bib{bouchut1993existence}{article}{
      author={Bouchut, Fran{\c{c}}ois},
       title={Existence and uniqueness of a global smooth solution for the
  {Vlasov-Poisson-Fokker-Planck} system in three dimensions},
        date={1993},
     journal={Journal of functional analysis},
      volume={111},
      number={1},
       pages={239\ndash 258},
}

\bib{carrillo2020quantitative}{inproceedings}{
      author={Carrillo, Jos{\'e}~A},
      author={Choi, Young-Pil},
       title={Quantitative error estimates for the large friction limit of
  {Vlasov} equation with nonlocal forces},
organization={Elsevier},
        date={2020},
   booktitle={Annales de l'institut henri poincar{\'e} c, analyse non
  lin{\'e}aire},
      volume={37},
       pages={925\ndash 954},
}

\bib{carrillo2021large}{article}{
      author={Carrillo, Jos{\'e}~A},
      author={Choi, Young-Pil},
      author={Peng, Yingping},
       title={Large friction-high force fields limit for the nonlinear
  vlasov--poisson--fokker--planck system},
        date={2021},
     journal={arXiv preprint arXiv:2103.12276},
}

\bib{carrillo2019propagation}{article}{
      author={Carrillo, Jos{\'e}~A},
      author={Choi, Young-Pil},
      author={Salem, Samir},
       title={Propagation of chaos for the {Vlasov--Poisson--Fokker--Planck}
  equation with a polynomial cut-off},
        date={2019},
     journal={Communications in Contemporary Mathematics},
      volume={21},
      number={04},
       pages={1850039},
}

\bib{carrillo2019aggregation}{incollection}{
      author={Carrillo, Jos{\'e}~A},
      author={Craig, Katy},
      author={Yao, Yao},
       title={Aggregation-diffusion equations: dynamics, asymptotics, and
  singular limits},
        date={2019},
   booktitle={Active particles, volume 2},
   publisher={Springer},
       pages={65\ndash 108},
}

\bib{carrillo2010asymptotic}{article}{
      author={Carrillo, Jos{\'e}~A},
      author={Fornasier, Massimo},
      author={Rosado, Jes{\'u}s},
      author={Toscani, Giuseppe},
       title={Asymptotic flocking dynamics for the kinetic {Cucker--Smale}
  model},
        date={2010},
     journal={SIAM Journal on Mathematical Analysis},
      volume={42},
      number={1},
       pages={218\ndash 236},
}

\bib{choi2020quantified}{article}{
      author={Choi, Young-Pil},
      author={Tse, Oliver},
       title={Quantified overdamped limit for kinetic vlasov-fokker-planck
  equations with singular interaction forces},
        date={2020},
     journal={arXiv preprint arXiv:2012.00422},
}

\bib{hui2021}{article}{
      author={Cipriani, Cristina},
      author={Huang, Hui},
      author={Qiu, Jinniao},
       title={Zero-inertia limit: from particle swarm optimization to
  consensus-based optimization},
        date={2021},
     journal={arXiv:2104.06939},
}

\bib{cucker2007emergent}{article}{
      author={Cucker, Felipe},
      author={Smale, Steve},
       title={Emergent behavior in flocks},
        date={2007},
     journal={IEEE Transactions on automatic control},
      volume={52},
      number={5},
       pages={852\ndash 862},
}

\bib{da2014stochastic}{book}{
      author={Da~Prato, Giuseppe},
      author={Zabczyk, Jerzy},
       title={Stochastic equations in infinite dimensions},
   publisher={Cambridge university press},
        date={2014},
}

\bib{duong2018quantification}{article}{
      author={Duong, Manh~Hong},
      author={Lamacz, Agnes},
      author={Peletier, Mark~A},
      author={Schlichting, Andr{\'e}},
      author={Sharma, Upanshu},
       title={Quantification of coarse-graining error in {Langevin} and
  overdamped {Langevin} dynamics},
        date={2018},
     journal={Nonlinearity},
      volume={31},
      number={10},
       pages={4517},
}

\bib{duong2017variational}{article}{
      author={Duong, Manh~Hong},
      author={Lamacz, Agnes},
      author={Peletier, Mark~A},
      author={Sharma, Upanshu},
       title={Variational approach to coarse-graining of generalized gradient
  flows},
        date={2017},
     journal={Calculus of variations and partial differential equations},
      volume={56},
      number={4},
       pages={1\ndash 65},
}

\bib{fetecau2019propagation}{article}{
      author={Fetecau, Razvan~C},
      author={Huang, Hui},
      author={Sun, Weiran},
       title={Propagation of chaos for the {Keller--Segel} equation over
  bounded domains},
        date={2019},
     journal={Journal of Differential Equations},
      volume={266},
      number={4},
       pages={2142\ndash 2174},
}

\bib{fetecau2015first}{article}{
      author={Fetecau, RC},
      author={Sun, Weiran},
       title={First-order aggregation models and zero inertia limits},
        date={2015},
     journal={Journal of Differential Equations},
      volume={259},
      number={11},
       pages={6774\ndash 6802},
}

\bib{freidlin2004some}{article}{
      author={Freidlin, Mark},
       title={Some remarks on the {Smoluchowski}--{Kramers} approximation},
        date={2004},
     journal={Journal of Statistical Physics},
      volume={117},
      number={3},
       pages={617\ndash 634},
}

\bib{godinho2015propagation}{inproceedings}{
      author={Godinho, David},
      author={Quininao, Cristobal},
       title={Propagation of chaos for a subcritical keller--segel model},
        date={2015},
   booktitle={Annales de l'ihp probabilit{\'e}s et statistiques},
      volume={51},
       pages={965\ndash 992},
}

\bib{ha2008particle}{article}{
      author={Ha, Seung-Yeal},
      author={Tadmor, Eitan},
       title={From particle to kinetic and hydrodynamic descriptions of
  flocking},
        date={2008},
     journal={Kinetic \& Related Models},
      volume={1},
      number={3},
       pages={415},
}

\bib{haile1993molecular}{article}{
      author={Haile, James~M},
      author={Johnston, Ian},
      author={Mallinckrodt, A~John},
      author={McKay, Susan},
       title={Molecular dynamics simulation: elementary methods},
        date={1993},
     journal={Computers in Physics},
      volume={7},
      number={6},
       pages={625\ndash 625},
}

\bib{holm2005aggregation}{article}{
      author={Holm, Darryl~D},
      author={Putkaradze, Vakhtang},
       title={Aggregation of finite-size particles with variable mobility},
        date={2005},
     journal={Physical review letters},
      volume={95},
      number={22},
       pages={226106},
}

\bib{hottovy2012noise}{article}{
      author={Hottovy, Scott},
      author={Volpe, Giovanni},
      author={Wehr, Jan},
       title={Noise-induced drift in stochastic differential equations with
  arbitrary friction and diffusion in the {Smoluchowski}-{Kramers} limit},
        date={2012},
     journal={Journal of Statistical Physics},
      volume={146},
      number={4},
       pages={762\ndash 773},
}

\bib{huang2020mean}{article}{
      author={Huang, Hui},
      author={Liu, Jian-Guo},
      author={Pickl, Peter},
       title={On the mean-field limit for the {Vlasov--Poisson--Fokker--Planck}
  system},
        date={2020},
     journal={Journal of Statistical Physics},
      volume={181},
      number={5},
       pages={1915\ndash 1965},
}

\bib{jabin2000macroscopic}{inproceedings}{
      author={Jabin, Pierre-Emmanuel},
       title={Macroscopic limit of vlasov type equations with friction},
organization={Elsevier},
        date={2000},
   booktitle={Annales de l'institut henri poincare (c) non linear analysis},
      volume={17},
       pages={651\ndash 672},
}

\bib{jabin2017mean}{incollection}{
      author={Jabin, Pierre-Emmanuel},
      author={Wang, Zhenfu},
       title={Mean field limit for stochastic particle systems},
        date={2017},
   booktitle={Active particles, volume 1},
   publisher={Springer},
       pages={379\ndash 402},
}

\bib{jacod2002limit}{book}{
      author={Jacod, Jean},
      author={Shiryaev, Albert},
       title={Limit theorems for stochastic processes},
   publisher={Springer Science \& Business Media},
        date={2002},
      volume={288},
}

\bib{keller1970initiation}{article}{
      author={Keller, Evelyn~F},
      author={Segel, Lee~A},
       title={Initiation of slime mold aggregation viewed as an instability},
        date={1970},
     journal={Journal of theoretical biology},
      volume={26},
      number={3},
       pages={399\ndash 415},
}

\bib{kramers1940brownian}{article}{
      author={Kramers, Hendrik~Anthony},
       title={{Brownian} motion in a field of force and the diffusion model of
  chemical reactions},
        date={1940},
     journal={Physica},
      volume={7},
      number={4},
       pages={284\ndash 304},
}

\bib{lazarovici2017mean}{article}{
      author={Lazarovici, Dustin},
      author={Pickl, Peter},
       title={A mean field limit for the {Vlasov--Poisson} system},
        date={2017},
     journal={Archive for Rational Mechanics and Analysis},
      volume={225},
      number={3},
       pages={1201\ndash 1231},
}

\bib{liu2019propagation}{article}{
      author={Liu, Jian-Guo},
      author={Yang, Rong},
       title={Propagation of chaos for the keller--segel equation with a
  logarithmic cut-off},
        date={2019},
     journal={Methods and Applications of Analysis},
      volume={26},
      number={4},
       pages={319\ndash 348},
}

\bib{motsch2014heterophilious}{article}{
      author={Motsch, Sebastien},
      author={Tadmor, Eitan},
       title={Heterophilious dynamics enhances consensus},
        date={2014},
     journal={SIAM review},
      volume={56},
      number={4},
       pages={577\ndash 621},
}

\bib{sznitman1991topics}{incollection}{
      author={Sznitman, Alain-Sol},
       title={Topics in propagation of chaos},
        date={1991},
   booktitle={Ecole d'\'{e}t\'{e} de probabilit\'{e}s de {Saint-Flour}
  {XIX}—1989},
   publisher={Springer},
       pages={165\ndash 251},
}

\bib{topaz2006nonlocal}{article}{
      author={Topaz, Chad~M},
      author={Bertozzi, Andrea~L},
      author={Lewis, Mark~A},
       title={A nonlocal continuum model for biological aggregation},
        date={2006},
     journal={Bulletin of mathematical biology},
      volume={68},
      number={7},
       pages={1601},
}

\bib{toscani2000one}{article}{
      author={Toscani, Giuseppe},
       title={One-dimensional kinetic models of granular flows},
        date={2000},
     journal={ESAIM: Mathematical Modelling and Numerical
  Analysis-Mod{\'e}lisation Math{\'e}matique et Analyse Num{\'e}rique},
      volume={34},
      number={6},
       pages={1277\ndash 1291},
}

\end{biblist}
\end{bibdiv}


\end{document}